\documentclass[11pt]{amsart}
\usepackage{amsmath, amssymb, amsthm}
\usepackage[margin=1in]{geometry}
\usepackage{mathrsfs} 
\usepackage{xcolor} 
\usepackage{hyperref}
\usepackage{dsfont}
\usepackage{comment}

\DeclareMathOperator{\supp}{supp}

\newcommand{\R}{\mathbb{R}}
\newcommand{\N}{\mathbb{N}}

\newtheorem{theorem}{Theorem}
\newtheorem{definition}{Definition}
\newtheorem{remark}{Remark}
\newtheorem{conjecture}{Conjecture}
\newtheorem{proposition}{Proposition}

\title[Fundamental solutions of the Logarithmic Laplacian]{Fundamental Solutions of the Logarithmic Laplacian: An Approach via the Division Problem}
\author{
  David Lee
}
\address{School of Mathematical and Physical Sciences, University of Technology Sydney (UTS)}
\email{davidchanwoo.lee@uts.edu.au}
\date{\today}

\begin{document}

\begin{abstract}
Existence of the fundamental solution of the logarithmic Laplacian (in dimensions $d \geq 3$) was established by Huyuan Chen and Laurent V\'eron (2024). In this note, we present an alternative approach, based on a modification on the classical division problem. This is inspired by the theory of fundamental solutions by Malgrange and Ehrenpreis. Moreover, we give a variant of the Liouville theorem for the logarithmic Laplacian and give some further clarification 
regarding a conjecture posed by Chen and V\'eron regarding the behavior of solutions in dimensions 1 and 2.
\end{abstract}

\subjclass[2010]{35A08, 35JXX, 35R11, 42B37, 42B10,42B20}
\thanks{Keywords: logarithmic Laplacian; fundamental solution; non-local}
\maketitle

\section{Introduction}

The logarithmic Laplacian is a nonlocal operator that arises as the formal derivative of the fractional Laplacian with respect to its exponent at zero. Namely, as introduced by Huyuan Chen and Tobias Weth~\cite{MR3995092} is given formally by

\begin{equation*}
\log(-\Delta)f:=\lim_{\sigma \rightarrow 0^{+}}\frac{1}{\sigma}\left ((-\Delta)^\sigma f-f \right ).
\end{equation*}
One can express the logarithmic Laplacian as
\begin{equation}\label{eqn:logarithmic_laplacian}
\log(-\Delta)f(x)=\gamma_{d}\int_{|x-y|\leq 1}\frac{f(x)-f(y)}{|x-y|^d}\,dy-\gamma_{d}\int_{|x-y|>1}\frac{f(y)}{|x-y|^d}\,dy+\rho_{d} f(x), \quad \text{for $x\in \mathbb{R}^d$, }
\end{equation}
where 
\begin{equation*}
\gamma_{d}=\frac{2}{\omega_{d-1}}, \quad \rho_{d}=2\log(2)+\psi(d/2)-\gamma_{E},
\end{equation*}
$\psi$ is the digamma function, $\omega_{d-1}$ is the surface area of $S^{d-1}$ (the unit sphere) and $\gamma_{E}$ is the Euler-Mascheroni constant. Furthermore, one can justify that the notation of $\log(-\Delta)$ is valid due to the identity via the Fourier transform $(\mathcal{F})$:
\begin{equation*}
	\mathcal{F}\left ( \log(-\Delta)f\right )(\xi)=2\log(|\xi|)\hat{f}(\xi), \quad \text{for $\xi \in \mathbb{R}^d\setminus \{0\}$,}
\end{equation*}
where $f\in \mathcal{S}(\mathbb{R}^d)$ (the space of Schwartz functions) and $\hat{f}$ is the Fourier transform of $f$. This non-local operator has importance in studying the asymptotic behavior of non-local elliptic problems involving the fractional Laplacian (with respect to the exponent), cf. \cite{MR3995092}, \cite{bungert2024convergenceratesfractionallocal} and references therein.

In the recent work \cite{MR4741033}, Huyuan Chen and Laurent V\'eron established the existence of the fundamental solution of the logarithmic Laplacian in dimensions $d\geq 3$. Namely, they solved the following problem:
\begin{equation*}
	\log(-\Delta)u=\delta_0, \quad \text{in $\R^d$}, 
\end{equation*}
where $ \delta_0 $ is the Dirac measure with mass at the origin.
We now present to the reader of the result of Huyuan Chen and Laurent V\'eron\ \cite{MR4741033}:
\begin{theorem}\label{thm:chen_veron}
	~\cite[Theorem 1.8]{MR4741033}
	Let $\mathcal{I}_{\alpha}$ be the Riesz potential operator (acting on a suitable class of measures) defined by
	\begin{equation*}
(\mathcal{I}_{\alpha}\mu)(x) = \frac{\Gamma\left (\frac{d-\alpha}{2}\right )}{\pi^{d/2}2^{\alpha}\Gamma(\alpha/2)}\int_{\mathbb{R}^d} \frac{\mu(dy)}{|x-y|^{d-\alpha}} \, dy, \quad \text{for $x\in \mathbb{R}^d$ and $\alpha \in (0,d/2)$.}
	\end{equation*}
	Let $d \geq 3$ and consider the following expression: 
\begin{equation}
\label{eqn:fundamental_solution_logarithmic_laplacian_chen}
\Phi_{\log}(x) = \int_0^1 (\mathcal{I}_{2t}\delta_0)(x) \, dt +\int_0^1 (\mathcal{I}_{2t}\Phi_1) (x) dt, \quad \text{for $x\in \mathbb{R}^d \setminus \{0\}$,}
\end{equation}
where $\Phi_1$ is a fundamental solution of the Helmholtz equation $ -\Delta u - u = \delta_0 $, given by
\begin{equation*}
\Phi_1(x) = \frac{i}{4} \left( \frac{1}{2\pi |x|} \right)^{\frac{d}{2} - 1} H^{(1)}_{\frac{d}{2} - 1}(|x|),\quad \text{for } x \in \mathbb{R}^d \setminus \{0\}.
\end{equation*}
Then, we have that  $\Phi_{\log}$ is a fundamental solution of the logarithmic Laplacian, i.e:
\begin{equation}\label{eqn:fundamental_solution_logarithmic_laplacian_old}
\log(-\Delta)\Phi_{\log}=\delta_0, \quad \text{in }\,\mathcal{D}'(\mathbb{R}^d),
\end{equation}
where $\mathcal{D}'(\R^d)$ is the classical space of distributions.

 Moreover, $ \Phi_{\log} $ is a (radial) locally integrable function which satisfies the following asymptotic behavior as $ |x| \to \infty $:
\begin{equation*}
|\Phi_{\log}(x)| \lesssim \frac{|x|^{\frac{3 - d}{2}}}{\log |x|} \quad \text{for } |x| \geq 2,
\end{equation*}
\end{theorem}
For the observed reader, one can notice that the fundamental solution considered by Chen and V\'eron, defined in~\eqref{eqn:fundamental_solution_logarithmic_laplacian_chen} has a interpretation which can be thought of as a variant of the Schwinger trick.

To remind the reader, the Schwinger trick is a technique that allows one to express the inverse of an operator as an integral involving a parameter. The most famous application of the Schwinger trick (although probably not originally due to Schwinger) being that the Coulomb kernel can be expressed as the integral of the heat kernel:
\begin{equation*}
		\frac{\Gamma\left(\frac{d}{2}-1\right)}{4\pi^{d/2}}|x|^{-(d-2)} = \int_0^\infty (4\pi t)^{-d/2} e^{-\frac{|x|^2}{4t}} \, dt, \quad \text{for $x\in \mathbb{R}^d\setminus \{0\}$}. 
\end{equation*}
Intuitively, the above identity can be reinpreted via the formal expresssion:
\begin{equation*}
	(-\Delta)^{-1}=\int_0^\infty e^{-t(-\Delta)} \, dt.
\end{equation*}
Nevertheless, it should be noted that this intuition breaks down in dimensions $1$ and $2$ due to the fact that the heat kernel (with respect to $t$) is not integrable in dimensions $1$ or $2$. 

To compare with the original problem considered by Chen and V\'eron, the proposed fundamental solution~\eqref{eqn:fundamental_solution_logarithmic_laplacian_chen} for the logarithmic Laplacian can also be obtained via a similar vein. Namely, one can obtain an expression for the fundamental solution via a formal identity utilizing the Riesz potential operator $\mathcal{I}_{2t}$:
\begin{equation*}
\frac{1}{\log(-\Delta)}=\int_0^1 (-\Delta)^{-t} \, dt +\int_0^1 (-\Delta)^{-t}((-\Delta)-1)^{-1} dt.
\end{equation*}
 Much like in the case of the Laplacian, it is not well-defined in dimensions $d=1,2$ since it becomes problematic to give meaning to the above expression. Despite this, there is a well established theory of fundamental solutions for linear partial differential operators with constant coefficients due to the Malgrange and Ehrenpreis theorem, cf.~\cite{MR86990,MR70048,MR68123,MR3995092} and~\cite{MR3379898} for a more recent treatement. 

In this note, \textbf{we present an alternative approach of obtaining the existence of a fundamental solution for the logarithmic Laplacian, based on a modification of the ideas of Malgrange and Ehrenpreis, that allows us to consider all dimensions}. Moreover, we give some insight into a conjecture raised by Chen and V\'eron, see Conjecture \ref{conjecture:chen_veron}. We also give a Liouville theorem for the logarithmic Laplacian. Of late, there has been considerable interest in classificiation of solutions for Helmholtz equations of non-local operators, cf.~\cite{MR4709301, MR4546886, MR4845366, MR3511811, hauer2023characterization, MR3477075}. In ~\cite{MR4709301}, Cheng, Li and Yang prove a Liouville theorem for a large class of non-local operators (under the assumption that the solution is bounded). Here, our proof of the Liouville theorem for the logarithmic Laplacian removes the assumption of boundedness but requires the use of Lizorkin distributions. Nevertheless, the classification of solutions greatly helps us in certain technical aspects of the proof. 

The fundamental idea utilized in this article is the following. Rather directly solving \eqref{def:fundamental_soln_log}, we apply the Fourier transform (in an appropriate way) to solve the following \emph{divison problem}:
\begin{equation*}
	\log(|\cdot|^2)\Psi=1.
\end{equation*}
A very naive solution for the above problem could be given by $\Psi(\xi):=\log(|\xi|^2)^{-1}$. However, the issue with this is that $\Psi$ is not a distribution (in the usual sense) on $\R^d$. It should be noted that the same situation arises when one considers the solvability of the Laplacian (in dimension 1 and 2). Observe that
\begin{equation}\label{eqn:division_problem_laplacian}
	(-\Delta)u=\delta_0 \iff |\cdot|^2\hat{u}=1.
\end{equation}
In dimension $1$ or $2$ one cannot simply take $\hat{u}=|\cdot|^{-2}$ since it is not a distribution. The way around this is to consider a renormalized version of $|\cdot|^{-2}$. In this note, we utilize a renormalized verion of of $\log(|\cdot|^2)^{-1}$ to obtain fundamental solutions of the logarithmic Laplacian. 

\begin{remark}
We should remind the reader that the theorem of Malgrange and Ehrenpreis doesn't quite apply to the logarithmic Laplacian since the symbol corresponding to the logarithmic Laplacian is not smooth at $0$. Moreover, it should be noted that the logarithmic Laplacian is generally not well-defined on the space of tempered distributions (we leave it as an exercise to the reader to justify to themselves that the logarithmic Laplacian of $1$ or any polynomial is not well-defined). For this reason, we consider an alternative notion of solution which utilizes the space of Lizorkin distributions.
\end{remark}
We introduce some notation and conventions before we consider our main result. 
\subsection{Notation and Conventions}

We use the following standard notation throughout:

\begin{itemize}
	\item $\mathcal{D}(\Omega)$: space of smooth, compactly supported functions on an open set $\Omega\subset \R^d$, \\
	$\mathcal{D}'(\Omega)$: its dual (distributions).
	\item $\mathcal{S}(\R^d)$: Schwartz space of rapidly decreasing smooth functions, \\
	$\mathcal{S}'(\R^d)$: tempered distributions.
	\item $\mathcal{F}$, $\mathcal{F}^{-1}$: Fourier and inverse Fourier transforms, normalized as
	\begin{equation*}
		\mathcal{F}(f)(\xi) = \int_{\R^d} e^{-ix\cdot\xi} f(x)\,dx, \qquad
		\mathcal{F}^{-1}(g)(x) = \frac{1}{(2\pi)^d} \int_{\R^d} e^{ix\cdot\xi} g(\xi)\,d\xi.
		\end{equation*}
	\item $S^{d-1}$: unit sphere in $\R^d$. 
\end{itemize}

All dualities are denoted by $\langle \cdot, \cdot \rangle$ with appropriate subscripts when needed. The Fourier transform is extended to distributions by duality. 

\subsection{Main Result}
Before we state the main result of this note we introduce the notion of solution that we will consider for the fundamental solution of the logarithmic Laplacian utilizing Lizorkin distributions. 

\begin{definition}[Lizorkin Space and Distributions]\label{def:distributions}

We define the Lizorkin space $\mathcal{Z}(\mathbb{R}^d)$ as the subspace of $\mathcal{S}(\mathbb{R}^d)$ given by
	\begin{equation*}
	\mathcal{Z}(\mathbb{R}^d)=\{\varphi\in \mathcal{S}(\mathbb{R}^d):\int_{\mathbb{R}^d}x^{\alpha}\varphi(x)\,dx=0, \text{ for all $\alpha \in \mathbb{N}^{d}_{0}$}\}.
	\end{equation*}
	Also, we denote 
	$\mathcal{Z}'(\mathbb{R}^d)$ as the set of continuous linear functionals on $\mathcal{Z}(\mathbb{R}^d)$. We call the space $\mathcal{Z}(\mathbb{R}^d)$ the Lizorkin space and $\mathcal{Z}'(\mathbb{R}^d)$ as the space of Lizorkin distributions. We also introduce
\begin{equation*}
		\Psi(\R^d):=\{\psi \in \mathcal{S}(\R^d):\partial^\alpha \psi(0)=0, \text{for all $\alpha \in \N^d_0$}\}.
	\end{equation*}
	We denote $\Psi'(\R^d)$ to be the space of linear continuous functionals on $\Psi(\R^d)$.
\end{definition}

\begin{remark}\label{rem:psi}
	As an immedidate consequence of the Fourier transform, we have that $\varphi \in \mathcal{Z}(\R^d)$ if and only if $\hat{\varphi}\in \Psi(\R^d)$. 
\end{remark}
\begin{definition}[Logarithmic Laplacian on the space of Lizorkin distributions]
	\label{def:log_laplacian_liz}
	We define the logarithmic Laplacian on the space of Lizorkin distributions $\log(-\Delta):\mathcal{Z}'(\mathbb{R}^d)\rightarrow\mathcal{Z}'(\mathbb{R}^d)$ as
		\begin{equation}\label{eqn:log_lizorkin}
			\langle \log(-\Delta)u,\varphi \rangle_{\mathcal{Z}'(\mathbb{R}^d),\mathcal{Z}(\mathbb{R}^d)}:=		\langle u,\log(-\Delta)\varphi \rangle_{\mathcal{Z}'(\mathbb{R}^d),\mathcal{Z}(\mathbb{R}^d)}, \quad \text{for all $(u,\varphi) \in (\mathcal{Z}'(\mathbb{R}^d),\mathcal{Z}(\mathbb{R}^d))$}.
		\end{equation}
\end{definition}

The reader might ask why we choose to utilize Lizorkin distributions as opposed to tempered distributions. It should be noted that one cannot define the logarithmic laplacian on the space of tempered distributions. To see this, observe that if one was to naively define 
\begin{equation*}
	\langle\log(-\Delta)u,\varphi\rangle_{\mathcal{S}'(\R^d),\mathcal{S}(\R^d)}:=	\langle u,\log(-\Delta)\varphi\rangle_{\mathcal{S}'(\R^d),\mathcal{S}(\R^d)},\quad \text{for $\varphi \in \mathcal{S}(\R^d)$},
\end{equation*}
one would run into an issue since 
\begin{equation*}
	\log(-\Delta)\varphi\notin \mathcal{S}(\R^d),
\end{equation*}
in general. It should be mentioned that polynomials are the problematic distributions for the logarithmic Laplacian. In turn, it requires for us to consider a notion of solution that allows us to work modulo polnomials. To be more precise, observe that  
	\begin{equation*}
	\langle u+ p,\varphi \rangle_{\mathcal{S}'(\R^d),\mathcal{S}(\R^d)}=        \langle u,\varphi \rangle_{\mathcal{S}'(\R^d),\mathcal{S}(\R^d)}, \quad \text{for all $(u,\varphi) \in \mathcal{S}'(\R^d)\times \mathcal{Z}(\R^d),$}
	\end{equation*}
	for all $p\in \mathcal{P}(\R^d)$. From this, we have that 
	$\mathcal{Z}'(\R^d)$ can be identified with the quotient space $\mathcal{S}'(\R^d)/\mathcal{P}(\R^d)$ and every element in $\mathcal{S}'(\R^d)$ can be identified as an element of $\mathcal{Z}'(\R^d)$ via projection. Every element in $\mathcal{Z}'(\R^d)$ can be uniquely identified (up to a polynomial) to an element in $\mathcal{S}'(\R^d)$. When we wish to emphasize that an element of $\mathcal{Z}'(\R^d)$ is an equivalence class we will use the notation $[u]$ as opposed to $u$. For further details on the Lizorkin space and distributions, we direct the reader to the book of Samko \cite{MR1918790}. 

We now give the notion of solution for the fundamental solution of the logarithmic Laplacian in this note. 

\begin{definition}
\label{def:fundamental_soln_log}
	We say that $E\in \mathcal{Z}'(\mathbb{R}^d)$ is a fundamental solution of the logarithmic Laplacian if it satisfies the following condition:
	\begin{equation}\label{eqn:fundamental_solution_logarithmic_laplacian_def}
		\log(-\Delta)E=\delta_0, \quad \text{in }\mathcal{Z}'(\mathbb{R}^d).
	\end{equation}
	Namely, we have that for all $\varphi \in \mathcal{Z}(\mathbb{R}^d)$, the following holds:
	\begin{equation*}
		\langle \log(-\Delta)E,\varphi \rangle_{\mathcal{Z}'(\mathbb{R}^d),\mathcal{Z}(\mathbb{R}^d)}=\langle \delta_0,\varphi \rangle_{\mathcal{Z}'(\mathbb{R}^d),\mathcal{Z}(\mathbb{R}^d)}=\varphi(0).	
	\end{equation*}

	Moreover, we say that $u\in \mathcal{Z}'(\R^d)$ is harmonic with respect to the logarithmic Laplacian if
	\begin{equation}\label{eqn:harmonic}
		\log(-\Delta)u=0, \quad \text{in }\mathcal{Z}'(\mathbb{R}^d).
	\end{equation}
	Specifically, for all $\varphi \in \mathcal{Z}(\mathbb{R}^d)$ the following holds:
	\begin{equation*}
		\langle \log(-\Delta)u,\varphi \rangle_{\mathcal{Z}'(\mathbb{R}^d),\mathcal{Z}(\mathbb{R}^d)}=0.	
	\end{equation*}

	We take the convention of identifying $[u]\in \mathcal{Z}'(\R^d)$ with the unique representative $u\in [u]$ such that $\hat{u}$ (the Fourier transform of $u$) can be identified as a locally integrable function in a sufficiently small enough neighbourhood around $0$. 
\end{definition}

\begin{remark}\label{rem:unique_extension}
	The last convention might seem a bit strange to the reader but this allows us to avoid a circumstance where polynomials are harmonic with respect to the logarithmic Laplacian. Moreover, this allows us to uniquely associate our solution with an element in $\mathcal{S}'(\R^d)$.
\end{remark}
As we mentioned before, we need to introduce a renormalized variant of $\log(|\cdot|^2)^{-1}$. There is a standard procedure to obtain such a distribution via Hadamard regularization, cf. the book of Sylvie Paycha \cite{MR2987296}. To save the reader some effort, we simply introduce a distribution that, a priori, we guess is a good candidate. 

\begin{definition}
	We define $\hat{E}_{\log}\in \mathcal{S}'(\R^d)$ as $\hat{E}_{\log}:=\hat{E}_{\log}^{1}+\hat{E}_{\log}^{2}$ where 
	\begin{equation*}
	\begin{split}
		\langle \hat{E}_{\log}^{1},\varphi\rangle_{\mathcal{S}'(\R^d),\mathcal{S}(\R^d)}
		&:=\frac{1}{2}\int_{||\xi|-1|<1}\frac{\varphi(\xi)-\varphi(\tfrac{1}{|\xi|}\xi)}{\log|\xi|}\,d\xi,\\
		\langle \hat{E}_{\log}^{2},\varphi\rangle_{\mathcal{S}'(\R^d),\mathcal{S}(\R^d)}
		&:=\frac{1}{2}\int_{|\xi|>2}\frac{\varphi(\xi)-\varphi(\tfrac{1}{|\xi|}\xi)}{\log|\xi|}\,d\xi,
	\end{split}
	\end{equation*}
	for $\varphi \in \mathcal{S}(\R^d)$. 
\end{definition}

\begin{remark}
	A modification of the standard process of Hadamard regularization, to obtain the renormalized variant of $(\log(|\cdot|^2))^{-1}$ is done via the following: 
	
	If $g$ is a function has singular behaviour near $|\xi|=1$ and we have that 
	\begin{equation*}
		\int_{||\xi|-1|>\epsilon}g(\xi)\,d\xi=a\log(\epsilon^{-1})+\sum_{k=1}^{n}b_{k}\epsilon^{-\lambda_{k}}+\Gamma(\epsilon),
	\end{equation*}
	for $a,b_{k},\lambda_{k}\in\mathbb{C}$ such that $\text{Re}(\lambda_{k})>0$ and $\lim_{\epsilon\rightarrow 0^{+}}\Gamma(\epsilon)$ is finite, then we call
	\begin{equation*}
		\mathcal{H}\int_{||\xi|-1|>\epsilon}g(\xi)\,d\xi:=\lim_{\epsilon\rightarrow 0^{+}}\Gamma(\epsilon). 
	\end{equation*}
	From this, we can define 
	\begin{equation*}
		\langle \textrm{P.f.}(\log(|\cdot|^2))^{-1},\varphi\rangle_{\mathcal{S}'(\R^d),\mathcal{S}(\R^d)}:=\mathcal{H}\int_{||\xi|-1|>\epsilon}\frac{\varphi(\xi)}{\log(|\xi|^2)}\,d\xi,
	\end{equation*}
	for $\varphi \in \mathcal{S}(\R^d)$. We choose not to do this since one gets $\hat{E}_{\log}$ with the addition of single layer distributions on the sphere (see Definition \ref{def:distri_sphere}) and it becomes a bit tedious to deal with. As we will see in Theorem \ref{thm:classification_logarithmic_laplacian} the Fourier inverse of these distributions end up being annihilated by the logarithmic Laplacian. 
\end{remark}

In addition to the above, to prove a Liouville theorem for the logarithmic Laplacian we need to introduce the space of distributions on $S^{d-1}$, the sphere in $\R^d$, and the space of single layer distributions on the sphere, cf. \cite[Chapter 6]{MR2000535}.

\begin{definition}\label{def:distri_sphere}\ \\
\textbf{Distributions on the sphere}: Let $\Lambda:\mathbb{R}^d\rightarrow\R$ be a smooth radial function such that $\supp(\Lambda)$ is contained in the open ball of radius 2 and $\Lambda|_{S^{d-1}}=1$. 
	For a measurable function $\tau:S^{d-1}\rightarrow \mathbb{C}$, we denote 
	\begin{equation*}
		(E\tau)(\xi):=\tau\left (\frac{1}{|\xi|}\xi\right )\Lambda(\xi), \quad \text{for $\xi\in \R^d$}. 
	\end{equation*}
	We say that $\tau \in \mathcal{D}(S^{d-1})$, if $E\tau\in \mathcal{D}(\R^d)$. 
	
	If we denote $\{p_{K,m}\}$ to be family of seminorms (on $\mathcal{D}(\R^d)$) defined by 
	\begin{equation*}
			p_{K,m}(\varphi) = \sup_{x \in K,\, |\alpha| \leq m} |\partial^\alpha \varphi(x)|,\quad \text{where $K \subset \R^d$ is compact and $m \in \mathbb{N}_0$}
	\end{equation*}
	then we endow $\mathcal{D}(S^{d-1})$ with the topology generated by the family of seminorms $\{\tilde{p}_{m}\}$ defined by 
	\begin{equation*}
			\tilde{p}_{m}(\tau) = p_{B(0,2),m}(E\tau),\quad \text{for $m\in \N_0$,}
	\end{equation*}
  where $B(0,2)$ is the open ball of radius 2. 
	
	We denote $\mathcal{D}'(S^{d-1})$ (the space of distributions on the sphere) to be the set of continuous linear functionals on $\mathcal{D}(S^{d-1})$.
	\ \\
\textbf{Single layer distributions on the sphere}: We say that $T\in \mathcal{D}'(\R^d)$ is a single layer distribution on the sphere (we denote the set of single layer distributions by $\mathcal{SL}(S^{d-1})$), if there exists a $\tilde{T}\in \mathcal{D}'(S^{d-1})$ such that 
	\begin{equation*}
		\langle T,\varphi\rangle_{\mathcal{D}'(\R^d),\mathcal{D}(\R^d)}=		\langle \tilde{T},\varphi|_{S^{d-1}}\rangle_{\mathcal{D}'(S^{d-1}),\mathcal{D}(S^{d-1})},
	\end{equation*}
	for all $\varphi \in \mathcal{D}(\R^d)$. 
\end{definition}

We now state the first result of this note. 

\begin{theorem}\label{thm:classification_logarithmic_laplacian}
	Let $d\geq 1$.  
	\begin{enumerate}
		\item \textbf{Liouville Theorem}:
			We have that $u$ is harmonic with respect to the logarithmic Laplacian if and only if $\hat{u}\in \mathcal{SL}(S^{d-1})$. 
		\item \textbf{Classification of fundamental solutions}:
		We have that $E$ solves \eqref{eqn:fundamental_solution_logarithmic_laplacian_def} if and only if $\hat{E}=\hat{E}_{\log}+v$ where $v\in \mathcal{SL}(S^{d-1})$. 
	\end{enumerate}
\end{theorem}

\begin{remark}
	An alternative way of stating the Liouville theorem for the above is that $u$ is harmonic with respect to the logarithmic Laplacian if and only if it is a (generalized) eigenfunction (with eigenvalue 1) of the Laplacian. Namely, 
	\begin{equation*}
		\log(-\Delta)u=0, \quad \iff \quad (-\Delta)u=u,
	\end{equation*}
	in $\mathcal{Z}'(\R^d)$. Given the assumption that we consider distributions which are locally integrable functions near $0$ (in frequency space), it becomes clear that $u$ is harmonic with respect to the logarithmic Laplacian if and only if $u$ is of the form 
	\begin{equation*}
	u(x)=\langle T,e_{x}\rangle_{\mathcal{D}'(S^{d-1}),\mathcal{D}(S^{d-1})},\quad \text{for $x\in \R^d$ and some $T\in \mathcal{D}'(S^{d-1})$.}
	\end{equation*}
  See also \cite[Theorem 4.1]{MR1730501}. 
\end{remark}

\begin{remark}
One might suspect that it is sufficient to prove the Liouville theorem by showing that if $u$ is harmonic with respect to the logarithmic Laplacian then showing that $\supp(\hat{u})\subset S^{d-1}$ would be sufficient. Unfortunately this isn't the case. For instance, one can easily verify that if $\hat{u}$ is the radial derivative of the surface measure of the sphere then it cannot be the case that $u$ is harmonic with respect to the logarithmic laplacian or is a generalized eigenfunction of the Laplacian despite the fact that its support is contained in the sphere. 
\end{remark}

The above result gives a comprehensive classification of all the fundamental solutions of the logarithmic Laplacian. However, it should be noted that the proposed solution of Chen and V\'eron satisfied some decay estimates. Namely, it was a solution with desirable properties, akin to the Sommerfield radiation condtion. It was conjectured by Chen and V\'eron that such solutions do exist. Here we present a modified version of the conjecture made by Chen and V\'eron.

\begin{conjecture}\cite{MR4741033}\label{conjecture:chen_veron}
	In dimension $1$ and $2$, we conjecture that the fundamental solution of the logarithmic Laplacian exists, is a (radial) locally integrally function, and satisfies the following bound:
	\begin{equation}\label{eqn:conj_bdd}
		|E(x)|\lesssim |x|^{-\left (\frac{d-1}{2}\right )}, \quad \text{for $|x|\geq 2$.}
	\end{equation}
\end{conjecture}

With this in mind, we state the other main result of this note. 

\begin{theorem}\label{thm:main_2}
	In dimensions $1$ and $2$, we have that there exists a fundamental solution to the logarithmic Laplacian such that it can be identified as a (radial) locally integrable function on $\{|x|\geq 2\}$ such that \eqref{eqn:conj_bdd} is satisfied. 
\end{theorem}

It should be noted that we haven't quite resolved Conjecture \ref{conjecture:chen_veron} since \textbf{we don't know if our fundamental solution is locally integrable near zero}. Nevertheless, we can still identify that our proposed solution can be identified with a locally integrable function away from zero which also satisfies \eqref{eqn:conj_bdd}. 
 
\begin{remark}
	It should be noted that the condition of local integrability wasn't originally made in the conjecture of Chen and V\'eron but we suspect that it is implied.
\end{remark}

\subsection{Outline}

We briefly outline the rest of this note. Section \ref{sec:derive_fundamental} is dedicated to the proof of Theorem~\ref{thm:classification_logarithmic_laplacian}. Section~\ref{sec:prop_soln} is dedicated to the proof of Theorem~\ref{thm:main_2}. Moreover, in Appendix \ref{appendix:log_lizorkin}, we show the reader the logarithmic Laplacian is well-defined on the space of Lizorkin distributions. 

\subsection{Acknowledgements}

The author is grateful for Tobias Weth for introducing him the article of Chen and V\'eron. The author also wants to give particular thanks to Ben Goldys for identifying an error in an earlier draft of this note.

The author also wants to thank Anthony Dooley and Huyuan Chen for helpful discussions. 

\section{Proof of Theorem~\ref{thm:classification_logarithmic_laplacian}}
\label{sec:derive_fundamental}

Let us briefly outline the proof for Theorem \ref{thm:classification_logarithmic_laplacian}. Via the Fourier transform, it is sufficient show:\\
\textbf{Existence of solution}: $\hat{E}_{\log}$ solves $\log(|\cdot|^2)\hat{E}=1$ in $\Psi'(\R^d)$,\\ 
	\textbf{Liouville theorem}: classify all solutions of $\log(|\cdot|^2)\hat{u}=0$ in $\Psi'(\R^d)$ which is equivalent to proving the first part of Theorem \ref{thm:classification_logarithmic_laplacian}. To do this, we do this in several steps:

	\paragraph{Step 1.} Show that $\supp(\hat{u})\subset S^{d-1}$. 
	\paragraph{Step 2.} Show that $\hat{u}$ is necessarily a single-layer distribution, see Definition \ref{def:distri_sphere}. To do this, it will be sufficient to show that 
	\begin{equation}\label{eqn:single_layer}
		\langle \hat{u},\varphi\rangle_{\mathcal{D}'(\R^d),\mathcal{D}(\R^d)}=	\langle \hat{u},E\tau\rangle_{\mathcal{D}'(\R^d),\mathcal{D}(\R^d)},
	\end{equation} 
	for all $\varphi\in \mathcal{D}(\R^d\setminus\{0\})$ such that $\varphi|_{S^{d-1}}=\tau$ (recall the definition of $E$ in Definition \ref{def:distri_sphere}). The reason why this is sufficient is that the distribution
	\begin{equation*}
		\langle T,\tau\rangle_{\mathcal{D}'(S^{d-1}),\mathcal{D}(S^{d-1})}:=		\langle \hat{u},E\tau\rangle_{\mathcal{D}'(\R^d),\mathcal{D}(\R^d)},\quad \text{for $\tau \in \mathcal{D}(S^{d-1})$},
	\end{equation*}
	is continuous on $\mathcal{D}(S^{d-1})$. To see this, if $\tau_n\rightarrow \tau\in \mathcal{D}(S^{d-1})$ then 
	\begin{equation*}
	\begin{split}
		&\langle T,\tau_n\rangle_{\mathcal{D}'(S^{d-1}),\mathcal{D}(S^{d-1})}=\langle \hat{u},E\tau_n\rangle_{\mathcal{D}'(\R^d),\mathcal{D}(\R^d)}\\
		&\quad \rightarrow \langle \hat{u},E\tau\rangle_{\mathcal{D}'(\R^d),\mathcal{D}(\R^d)}=\langle T,\tau\rangle_{\mathcal{D}'(S^{d-1}),\mathcal{D}(S^{d-1})},
	\end{split} 
	\end{equation*}
	hence $T$ is continuous on $\mathcal{D}(S^{d-1})$. 
	\paragraph{Step 3.} We verify that if $\hat{u}$ is a single layer distribution on the sphere (corresponding to $T\in \mathcal{D}'(S^{d-1})$) then 
	$\log(-\Delta)u=0$ in $\mathcal{Z}'(\R^d)$. 

We now proceed with the proofs. 

\begin{proof}[Proof of Existence]
	By direct computation, we have that 
	\begin{equation*}
		\begin{split}
			&\langle \log(|\cdot|^2)\hat{F}_{\log},\psi \rangle_{\Psi'(\R^d),\Psi(\R^d)}\\
			&=\langle \hat{F}_{\log},\log(|\cdot|^2)\psi \rangle_{\Psi'(\R^d),\Psi(\R^d)},\\
			&=\frac{1}{2}\int_{||\xi|-1|<1}\left (\frac{1}{\log|\xi|}\right )\log(|\xi|^2)\psi(\xi)\,d\xi+\frac{1}{2}\int_{||\xi|-1|>1}\frac{\log(|\xi|^2)\psi(\xi)}{\log(|\xi|)}\,d\xi,\\
			&=\int_{\R^d}\psi(\xi)\,d\xi=\langle 1,\psi\rangle_{\Psi'(\R^d),\Psi(\R^d)}, 
		\end{split}
	\end{equation*}
	for all $\psi \in \Psi(\R^d)$.
\end{proof}

\begin{proof}[Proof of the Liouville theorem]
We proceed by following the steps from before. 
\begin{enumerate}
	
	\item[Step 1.]

We show that $\supp(\hat{u})\subset S^{d-1}$. To prove this, it will be sufficent to show that 
\begin{equation*}
	\langle \hat{u},\psi\rangle_{\mathcal{D}'(\R^d),\mathcal{D}(\R^d)}=0,
\end{equation*}
for $\psi \in \mathcal{D}(\R^d\setminus(S^{d-1}\cup \{0\}))$. 
Note that 
\begin{equation*}
	\psi=\log(|\cdot|^2)\left (\frac{1}{\log(|\cdot|^2)}\psi\right ),
\end{equation*}
and $\tilde{\psi}:=\frac{1}{\log(|\cdot|^2)}\psi$ also belongs to $\mathcal{D}(\R^d\setminus(S^{d-1}\cup \{0\}))$. Hence, since 
\begin{equation*}
\langle \hat{u},\log(|\cdot|^2)\psi\rangle_{\mathcal{S}'(\R^d),\mathcal{S}(\R^d)}=0, \quad \text{for all $\psi \in \Psi(\R^d)$}, 
\end{equation*}
we have that $\supp(\hat{u})\subset S^{d-1}$. 
\item[Step 2.] 
Let $\psi\in \mathcal{D}(\R^d\setminus\{0\})$ such that $\psi|_{S^{d-1}}=\tau$. Note that, 

\begin{equation*}
\begin{split}
	\psi(\xi)-(E\tau)(\xi)&=\varphi(\xi)-\Lambda(\xi)\psi\left ( \frac{1}{|\xi|}\xi\right)\\
	&=(1-\Lambda(\xi))\psi(\xi)+\Lambda(\xi)\left (\psi(\xi)-\psi\left ( \frac{1}{|\xi|}\xi\right)\right ), \quad \text{for $\xi\in \R^d$}. 
\end{split}
\end{equation*} 
Given that $\supp(\hat{u})\subset S^{d-1}$, we have that 
\begin{equation*}
	\langle\hat{u},(1-\Lambda)\psi\rangle_{\mathcal{D}'(\R^d),\mathcal{D}(\R^d)}=0.
\end{equation*}
We show that 
\begin{equation*}
	\langle\hat{u},\tilde{\psi}\rangle_{\mathcal{D}'(\R^d),\mathcal{D}(\R^d)}=0,
\end{equation*}
where
\begin{equation*}
\tilde{\psi}(\xi):=\Lambda(\xi)\left (\psi(\xi)-\psi\left ( \frac{1}{|\xi|}\xi\right)\right ), \quad \text{for $x\in \R^d$}. 
\end{equation*}
Note that 
\begin{equation*}
\begin{split}
	\tilde{\psi}(\xi)&:=\Lambda(\xi)\left (\psi(\xi)-\psi\left ( \frac{1}{|\xi|}\xi\right)\right )\\
	&=\log(|\xi|^2)\Lambda(\xi)\left (\frac{\psi(\xi)-\psi\left ( \frac{1}{|\xi|}\xi\right)}{\log(|\xi|^2)}\right ),\\
	&=\frac{1}{2}\log(|\xi|^2)\Lambda(\xi)\left (\frac{\psi(\xi)-\psi\left ( \frac{1}{|\xi|}\xi\right)}{|\xi|-1}\right )\left (\frac{|\xi|-1}{\log(|\xi|)}\right ).
\end{split}
\end{equation*}
Since $\supp(\Lambda)$ contains an open neighbourhood of $S^{d-1}$ (but not the point $0$), it is sufficient to show that 
\begin{equation*}
	\left (\frac{\psi(\xi)-\psi\left ( \frac{1}{|\xi|}\xi\right)}{|\xi|-1}\right ), \quad \frac{|\xi|-1}{\log(|\xi|)}
\end{equation*}
are smooth (alternatively have smooth extensions) for $A_{\epsilon}:=\{\xi\in \R^d:\epsilon<|\xi|<2\}$ for some sufficently small $\epsilon\in (0,1)$. 

 We show that $\xi\mapsto \frac{\psi(\xi)-\psi\left ( \frac{1}{|\xi|}\xi\right)}{|\xi|-1}$ is smooth on $A_{\epsilon} \setminus \{|\xi|=1\}$. By the Taylor theorem, we have that 
	\begin{equation*}
	\begin{split}
		\psi(\xi)-\psi\left ( \frac{1}{|\xi|}\xi\right )&=\left (\xi-\frac{1}{|\xi|}\xi\right )\cdot \int_0^1\nabla \psi \left( \frac{1}{|\xi|}\xi+(1-t)\left (\xi-\frac{1}{|\xi|}\xi\right )\right )\,dt,\\
		\implies \frac{\psi(\xi)-\psi\left ( \frac{1}{|\xi|}\xi\right )}{|\xi|-1}&=\frac{1}{|\xi|}\xi\cdot \int_0^1\nabla \psi \left( \frac{1}{|\xi|}\xi+(1-t)\left (\xi-\frac{1}{|\xi|}\xi\right )\right )\,dt. 
	\end{split}
	\end{equation*} 
	From this, it is clear that $x\in A_{\epsilon}\setminus\{|\xi|=1\}\mapsto \frac{\psi(\xi)-\psi\left ( \frac{1}{|\xi|}\xi\right )}{|\xi|-1}$ 
	has a smooth extension to $A_{\epsilon}$. 

	We now show that $\xi\in A_{\epsilon}\setminus\{|\xi|=1\}\mapsto \frac{|\xi|-1}{\log(|\xi|)}$ has a smooth extension to $A_\epsilon$. We do this by considering the complex valued function 
	\begin{equation*}
		\omega(z):=\frac{z-1}{\text{Log}(z)}, \quad \text{for $z\in \mathbb{C}\setminus ((-\infty,0]\cup\{1\})$},
	\end{equation*}
    where $\text{Log}$ is the principal branch of the complex logarithm.
	It is clear that $\omega$ is complex differentiable (hence holomorphic) on $\mathbb{C}\setminus ((-\infty,0]\cup\{1\})$. Moreover, since 
	\begin{equation*}
		\lim_{z\rightarrow 1}(z-1)\omega(z)=0,
	\end{equation*}
	we have that $\omega$ has a holomorphic extension to the point $1$. Since $\omega$, the extension, is holomorphic we have that 
	\begin{equation*}
		\xi\in A_{\epsilon}\mapsto \omega(|\xi|)=\frac{|\xi|-1}{\log(|\xi|)},
	\end{equation*} 
	is smooth. 

To summarize, we have that 
\begin{equation*}
	\tilde{\psi}(\xi)
	=\frac{1}{2}\log(|\xi|^2)\underset{\textnormal{smooth with compact support}}{\Lambda(\xi)}
		\underset{\textnormal{smooth}}{\left (\frac{\psi(\xi)-\psi\left ( \frac{1}{|\xi|}\xi\right)}{|\xi|-1}\right )}
		\underset{\textnormal{smooth}}{\left (\frac{|\xi|-1}{\log(|\xi|)}\right )},
\end{equation*}
hence, since 
 	\begin{equation*}
\langle \hat{u},\log(|\cdot|^2)\psi\rangle_{\mathcal{S}'(\R^d),\mathcal{S}(\R^d)}=0, \quad \text{for all $\psi \in \Psi(\R^d)$}. 
\end{equation*}
we must have that 
\begin{equation*}
\langle \hat{u},\tilde{\psi}\rangle_{\mathcal{D}'(\R^d),\mathcal{D}(\R^d)}=0,
\end{equation*}
which shows that 
\begin{equation*}
	\langle \hat{u},\psi\rangle_{\mathcal{D}'(\R^d),\mathcal{D}(\R^d)}= 	\langle \hat{u},E\tau\rangle_{\mathcal{D}'(\R^d),\mathcal{D}(\R^d)},
\end{equation*}
for all $\psi\in \mathcal{D}(\R^d\setminus\{0\})$ such that $\psi|_{S^{d-1}}=\tau$. Hence, we must have that $\hat{u}$ is necessarily a single-layer distribution. 
\item[Step 3.] It is sufficient for us to show that 
\begin{equation*}
	\langle \hat{u},\log(|\cdot|^2)\psi\rangle_{\mathcal{D}'(\R^d),\mathcal{D}(\R^d)}=0
\end{equation*}
for all $\psi \in \Psi(\R^d)$ with compact support since $\supp(\hat{u})$ is compact. Since we assume that $\hat{u}$ is a single layer distribution, we must have 
\begin{equation*}
	\langle\hat{u},\log(|\cdot|^2)\psi\rangle_{\mathcal{D}'(\R^d),\mathcal{D}(\R^d)}=\langle T,\log(|\cdot|^2)\psi|_{S^{d-1}}\rangle_{\mathcal{D}'(S^{d-1}),\mathcal{D}(S^{d-1})}=0,
\end{equation*}
since $\log(|\cdot|^2)\psi|_{S^{d-1}}=0$. In turn, we have that $u$ solves $\log(-\Delta)u=0$ in $\mathcal{Z}'(\R^d)$. 

\end{enumerate}
\end{proof}
\section{Proof of Theorem \ref{thm:main_2}}
\label{sec:prop_soln}

One potential strategy to prove Theorem \ref{thm:main_2} is to directly take the Fourier inverse of $\hat{E}_{\log}$. This may work but this would require tricky manipulations. The author hasn't been able to succeed with this approach since one has to bound highly oscillatory integrals. 

The strategy that we employ is, based on Theorem \ref{thm:classification_logarithmic_laplacian}, is that we choose an appropriate element in $\mathcal{SL}(S^{d-1})$ that allows us to compare with the classical Helmholtz solution. This greatly simplifies things and allows us to remove certain oscillatory integrals. Let us now summarize this. 

\paragraph{Step 1.}

We show that there exists a fundamental solution of the logarithmic Laplacian of the form 
\begin{equation*}
	E=\Phi+E_{\text{rem}}^{1}+E_{\text{rem}}^{2}, 
\end{equation*}
where 
        \begin{equation*}
            \begin{split}
                \langle \hat{E}^1_{\text{rem}},\varphi\rangle_{\mathcal{S}'(\R^d),\mathcal{S}(\R^d)}&:=\int_{||\xi|-1|<1}\left (\frac{1}{\log(|\xi|^2)}-\frac{1}{|\xi|^2-1}\right )\varphi(\xi)\,d\xi,\\
                \langle \hat{E}^2_{\text{rem}},\varphi\rangle_{\mathcal{S}'(\R^d),\mathcal{S}(\R^d)}&:=\int_{|\xi|>2}\left (\frac{1}{\log(|\xi|^2)}-\frac{1}{|\xi|^2-1}\right )\varphi(\xi)\,d\xi,
            \end{split}
        \end{equation*}
        for $\varphi \in\mathcal{S}(\R^d)$, and 
\begin{equation}\label{eqn:helmholtz_soln}
        \Phi(x)=
        \begin{cases}
            \displaystyle\frac{-i}{2}e^{i|x|}, \quad &\text{for $d=1$ and $x\in \R$, }\\
            \displaystyle\frac{(2\pi)^{-(d-2)/2}}{4i}\frac{H^{(1)}_{(d-2)/2}(|x|)}{|x|^{\tfrac{d-2}{2}}}, \quad &\text{for $d>1$ and $x\in \R^{d}\setminus \{0\}$.}
        \end{cases}
        \end{equation}

$H^{(1)}$ is the usual Hankel function, see \cite[Chapter 14, Section 10]{MR3887685}. Moreover, $\Phi$ is the classical Helmholtz solution, ie it solves $(-\Delta)u-u=\delta_0$ in $\mathcal{S}'(\R^d)$. 

This is achieved through comparing the fundamental solutions of the logarithmic Laplacian (in Theorem \ref{thm:classification_logarithmic_laplacian}) with those of the classical Helmholtz problem. 

Note that to check that $\hat{E}^1_{\text{rem}}$ is actually a tempered distribution it should be observed, by considering the Laurent series of $\tfrac{1}{\text{Log}(z)}$ where $\text{Log}$ is the principal branch of the complex logarithm, then we have that 
\begin{equation*}
	\{||\xi|-1|<1\}\mapsto \frac{1}{\log(|\xi|^2)}-\frac{1}{|\xi|^2-1},
\end{equation*}
is continuous (or at least has a continuous extension).

\paragraph{Step 2.}

With our choice of solution $E=\Phi+E_{\text{rem}}^{1}+E_{\text{rem}}^{2}$, we verify the bound stated in Theorem \ref{thm:main_2}. 

We can quickly verify that both $\Phi$ and $E_{\text{rem}}^{1}$ satisfy the relevant bounds. For $\Phi$ this can be verified directly and for $E_{\text{rem}}^{1}$ this follows since it is the Fourier inverse of a bounded radial function which has compact support. 

It remains to show that $E_{\text{rem}}^{2}$ can be identified as a locally integrable function on $\{|x|\geq 2\}$ and satisfies the relevant bounds. This means that we only need to show that $E_{\log}^{2}$ and $E_{\text{Helm}}^{2}$ satisfies the relevant bounds (as well identifying as a locally integrable function away from $0$). 
\begin{itemize}
\item[Step 1.]
We introduce the following distributions:
\begin{equation*}
        \begin{split}
            \langle \hat{E}_{\text{Helm}}^{1},\varphi \rangle_{\mathcal{S}'(\R^d),\mathcal{S}(\R^d)}&:=\int_{||\xi|-1|<1}\frac{\varphi(\xi)-\varphi(\frac{1}{|\xi|}\xi)}{|\xi|^2-1}\,d\xi,\\
            \langle \hat{E}_{\text{Helm}}^{2},\varphi \rangle_{\mathcal{S}'(\R^d),\mathcal{S}(\R^d)}&:=\int_{|\xi|>2}\frac{\varphi(\xi)}{|\xi|^2-1}\,d\xi,
        \end{split}
\end{equation*}
    for $\varphi \in \mathcal{S}(\R^d)$. 

Note that $w$ solves the classical Helmholtz problem:
        \begin{equation}\label{eqn:helmholtz}
        (-\Delta)w-w=\delta_0, \quad \text{in $\mathcal{S}'(\R^d).$} 
    \end{equation}
	if and only if $\hat{w}$ is of the form 
    \begin{equation*}
        \hat{w}=\hat{E}_{\text{Helm}}^{1}+\hat{E}_{\text{Helm}}^{2}+v,
    \end{equation*}
    where 
 and $v\in \mathcal{SL}(S^{d-1})$, cf. \cite[Exemple, page 128]{MR209834}. Note that this classification isn't that much different from Theorem~\ref{thm:classification_logarithmic_laplacian}.  
	From Theorem \ref{thm:classification_logarithmic_laplacian}, we have that if $E$ is a fundamental solution to the logarithmic Laplacian (more specifically, the unique extension from $\mathcal{Z}'(\R^d)$ to $\mathcal{S}'(\R^d)$, cf. Definition ~\ref{def:fundamental_soln_log} and Remark~\ref{rem:unique_extension}) then we have that
\begin{equation*}
        \begin{split}
		\hat{E}&=\hat{E}^1_{\log}+\hat{E}^2_{\log}+v,\\
		&=\hat{E}^1_{\text{Helm}}+\hat{E}^2_{\text{Helm}}+v+(\hat{E}^1_{\log}-\hat{E}^1_{\text{Helm}})+(\hat{E}^2_{\log}-\hat{E}^2_{\text{Helm}}).  
		\end{split}
\end{equation*}
Moreover, we have that 
\begin{equation}\label{eqn:important_comp}
    \begin{split}
        &\langle\hat{E}^{1}_{\log}-\hat{E}^{1}_{\text{Helm}},\varphi\rangle_{\mathcal{S}'(\R^d),\mathcal{S}(\R^d)}\\
    &=\int_{||\xi|-1|<1}\left (\frac{\varphi(\xi)-\varphi(\frac{1}{|\xi|}\xi)}{\log(|\xi|^2)}\right )-\left (\frac{\varphi(\xi)-\varphi(\frac{1}{|\xi|}\xi)}{|\xi|^2-1}\right )\,d\xi,\\
    &=\langle \hat{E}_{\text{rem}}^{1},\varphi\rangle_{\mathcal{S}'(\R^d),\mathcal{S}(\R^d)}+\int_{||\xi|-1|<1}\left (\frac{1}{\log(|\xi|^2)}-\frac{1}{|\xi|^2-1}\right )\varphi\left (\frac{1}{|\xi|}\xi\right )\,d\xi.
    \end{split}
\end{equation}
for $\varphi \in \mathcal{S}(\R^d)$. The distribution given by 
\begin{equation*}
    \varphi \in \mathcal{S}(\R^d)\mapsto \int_{||\xi|-1|<1}\left (\frac{1}{\log(|\xi|^2)}-\frac{1}{|\xi|^2-1}\right )\varphi\left (\frac{1}{|\xi|}\xi\right )\,d\xi,
\end{equation*}
happens to be a single layer distribution on the sphere. 
In turn, we have that $E$ is a fundamental solution of the logarithmic Laplacian, if and only if 
\begin{equation*}
	\hat{E}=\hat{E}^1_{\text{Helm}}+\hat{E}^2_{\text{Helm}}+\overline{v}+\hat{E}^1_{\text{rem}}+\hat{E}^2_{\text{rem}},
\end{equation*}
		
In turn, we know that there exists a $\overline{v}\in \mathcal{SL}(S^{d-1})$ such that $\hat{\Phi}=\hat{E}^1_{\text{Helm}}+\hat{E}^2_{\text{Helm}}+\overline{v}$ where $\Phi$ is given in~\eqref{eqn:helmholtz_soln}. This shows that there exists a solution of the form ${E}=\Phi+{E}^1_{\text{rem}}+{E}^2_{\text{rem}}$.
 \item[Step 2.]
 We verify the bounds of $E^{2}_{\text{Helm}}$ and $E^{2}_{\log}$ individually. 
\paragraph{$E^{2}_{\text{Helm}}$:}
In dimension 1, this is the Fourier inverse of a function in $L^1(\R)$ hence it is bounded and satisfies the bound needed. In dimension 2, we can approximate $E^{2}_{\text{Helm}}$ by $E^{2,N}_{\text{Helm}}$ which is the Fourier inverse of $\mathds{1}_{\{2<|\xi|<N\}}(|\cdot |^2-1)^{-1}$ and take the limit $N\rightarrow \infty$ in $\mathcal{S}'(\R^2)$. 

By direct computation, we have that
\begin{equation*}
    {E}_{\text{Helm}}^{2,N}(x)=(2\pi)^{-2}\int_{2}^{N}rJ_0(r|x|)\frac{\,dr}{r^2-1}, \quad \text{for $x\in \R^2, N>2$,}
\end{equation*}
where $J$ is the Bessel function of the first kind. 

Since $|J_{0}(r|x|)|\lesssim r^{-1/2}|x|^{-1/2}$, see \cite[Appendix B.8]{MR3243734}, we have that 
\begin{equation*}
    (2\pi)^{2}F^{2}_{\text{Helm}}(x)=\int_{2}^{\infty}rJ_0(r|x|)\frac{dr}{r^2-1}, \quad \text{for $x\ne 0$,}
\end{equation*}
and $|E^{2}_{\text{Helm}}(x)|\lesssim |x|^{-1/2}$ for $|x|>2$ which verifies the bound needed. 

\paragraph{$E^{2}_{\log}$:}

We focus on the proof in dimension 2 since the proof in dimension 1 is very similar. Again, we approximate $E^{2}_{\log}$ by $E^{2, N}_{\log}$, which is the Fourier inverse of $\mathds{1}_{\{2<|\xi|<N\}}(\log(|\cdot|^2))^{-1}$, and we take the limit $N\rightarrow \infty$ (in $\mathcal{S}'(\R^d)$) of $|\cdot|^2E^{2,N}_{\log}\rightarrow |\cdot|^2E^{2}_{\log}$ (note that multiplication with polynomials in continuous in $\mathcal{S}'(\R^d)$). 

By 2 applications of integration by parts, we have that 
\begin{equation*}
	\begin{split}
		(2\pi)^2{E}_{\log}^{2,N}(x)
		=&\int_{2}^{N}rJ_0(r|x|)\frac{\,dr}{\log(r)},\\
			=& \left (\frac{N}{\log(N)}\frac{J_1(N|x|)}{|x|}- \frac{2}{\log(2)}\frac{J_1(2|x|)}{|x|}\right ) - \frac{1}{|x|}\int_2^N \frac{J_1(r|x|)}{ (\log r)^2} dr,\\
			=& \underbrace{\left (\frac{N}{\log(N)}\frac{J_1(N|x|)}{|x|}- \frac{2}{\log(2)}\frac{J_1(2|x|)}{|x|}\right )}_{:=G_{N}^{1}}\\
			&+\underbrace{\left (\frac{J_0(2|x|)}{\log(2)^2|x|^2}-\frac{J_0(N|x|)}{\log(N)^2|x|^2} - \frac{1}{|x|^2}\int_2^N \frac{J_1(r|x|)}{ r(\log r)^3} dr\right )}_{:=G_{N}^{2}},\\
	\end{split}
\end{equation*}
for $x\ne 0$ and $N>2$. It is straightforward to check that $G^{1}:=\lim_{N\rightarrow \infty}G^{1}_{N}$, in $\mathcal{S}'(\R^2)$, is a locally integrable function of the form:

\begin{equation*}
	G^{1}(x)=- \frac{2}{\log(2)}\frac{J_1(2|x|)}{|x|}, \quad \text{for $x\in \R^d\setminus\{0\}$}.
\end{equation*}
Similarly, one can verify that $|\cdot|^2G_{N}^{2}\rightarrow |\cdot|^2G^{2}$ and $|\cdot|^2G^{2}$ is a (bounded) regular distribution of the form 
\begin{equation*}
	|x|^2G^{2}(x)=\frac{J_0(2|x|)}{\log(2)^2} - \int_2^\infty \frac{J_1(r|x|)}{ r(\log r)^3} dr, \quad \text{for $x\in \R^2$.}
\end{equation*}
In turn, we have that $G_{2}$ can be identified as a regular distribution (as a element of $\mathcal{D}'(\{|x|\geq 2\})$) which satisfies $|G^{2}(x)|\lesssim |x|^{-2}$ for $|x|\geq 2$. In turn, by considering the decay of $G^{1}$ and $G^2$ we have our result. 

\end{itemize}
\begin{remark}
	Note that what is stopping us from fully resolving Conjecture~\ref{conjecture:chen_veron} is by identifying $G^2$ (as defined in the previous proof) as a regular distribution on the whole of $\R^d$. We can only do this for $|\cdot |^2G^2$. It would be very interesting if one can resolve whether or not this is a regular distribution. 
\end{remark}

\section*{Appendix: Logarithmic Laplacian on the space of Lizorkin distributions}
\addcontentsline{toc}{section}{Appendix: Logarithmic Laplacian on the space of Lizorkin distributions}
\label{appendix:log_lizorkin}
The following proposition shows that the logarithmic Laplacian is well-defined on the space of Lizorkin distributions. 
\begin{proposition}
	The logarithmic Laplacian $\log(-\Delta):\mathcal{Z}'(\R^d)\rightarrow \mathcal{Z}'(\R^d)$, as given in Definition~\ref{def:log_laplacian_liz}, is well-defined. 
\end{proposition}
\begin{proof}
	It is sufficient to show that 
	\begin{equation*}
		\log(|\cdot|^2)\psi\in \Psi(\R^d),
	\end{equation*}
	if $\psi\in \Psi(\R^d)$ since it ensures that $\log(-\Delta)\varphi \in \mathcal{Z}(\R^d)$ for all $\varphi\in \mathcal{Z}(\R^d)$, see Remark \ref{rem:psi}. By $\log(|\cdot|^2)\psi$ we mean it's smooth extension at $0$ which we show exists.
	
	Note that 
	\begin{equation*}
		\log(|\cdot|^2)\psi=		|\cdot|^{2m}\log(|\cdot|^2)|\cdot|^{-2m}\psi, \quad 
	\text{for all $m\in \N$.} 
	\end{equation*} 
	Note that the function
	\begin{equation*}
		\lambda \mapsto \lambda^m\log(\lambda), \quad \text{for $\lambda\in (0,\infty)$},
	\end{equation*}
	has a smooth extension at $0$ (up to $m-1$ derivatives). Moreover, the function and it's derivatives (up to $m-1$ derivatives) grow at an algebraic rate. In turn, we have that $|\cdot|^{2m}\log(|\cdot|^2)$ has smooth derivatives up to order $m-1$ which grow at an algebraic rate. 
	
	The derivatives of $|\cdot|^{-2m}\psi$ decay at a rapid rate. Moreover, by the Taylor remainder theorem, we have that 
	\begin{equation*}
	|\psi(\xi)|\leq |\xi|^n,\quad \text{for $|\xi|\leq 1$ and $n\in \N$. }
	\end{equation*}
	Hence, we can deduce by the product rule, that $|\cdot|^{-2m}\psi$ has a smooth extension at $0$. In turn, since $m$ is arbitrary, we have that 
		\begin{equation*}
	\log(|\cdot|^2)\psi\in \Psi(\R^d).
	\end{equation*}
\end{proof}
\bibliographystyle{alpha}

\end{document}